\documentclass[10pt]{amsart}
\usepackage{latexsym, amsmath, amssymb}
\usepackage{version}
\usepackage{epsfig,graphics}
\usepackage{color}
\usepackage[T1]{fontenc}
\usepackage{amsthm, amsopn, amsfonts}

\newcommand{\cal}{\mathcal}

\newcommand{\newsection}[1]
{\section{#1}\setcounter{theorem}{0} \setcounter{equation}{0}
\par\noindent}

\newtheorem{theorem}{Theorem}

\newtheorem{lemma}[theorem]{Lemma}

\newtheorem{proposition}[theorem]{Proposition}

\newtheorem{remark}[theorem]{Remark}

\newcommand{\Z}{{\mathbb Z}}
\newcommand{\M}{\cal M}
\newcommand{\R}{{\mathbb R}}
\newcommand{\tv}{{\tilde{v}}}
\newcommand{\rs}{{r^*}}
\newcommand{\tphi}{{\tilde{\phi}}}
\newcommand{\ang}{{\not\negmedspace\nabla}}
\renewcommand{\S}{{\mathbb S}}

\begin{document}

\title
{ Strichartz estimates on Kerr black hole backgrounds }

\author{Mihai Tohaneanu}

\address{Department of Mathematics, Purdue University,
  West Lafayette, IN 47907-2067}

\begin{abstract}
 We study the dispersive properties for the wave equation in the Kerr
 space-time with small angular momentum. The main result of this
 paper is to establish Strichartz estimates for solutions of the aforementioned equation. This follows a
 local energy decay result for the Kerr space-time obtained in the earlier work 
\cite{TT} and uses the techniques and results from \cite{MMTT} by the author and collaborators. As an   application, we then prove global well-posedness and uniqueness for the energy critical semilinear wave equation.

\end{abstract}

\maketitle
  
 \newsection{Introduction}
 
  Understanding the decay properties of solutions to the linear wave equation on Kerr backgrounds is considered a crucial first step in proving the stability of the Kerr solution to the Einstein equations. Until recently even the problem of obtaining uniform bounds for such solutions was completely open, and 
only some partial results (pointwise decay and energy bounds for azimuthal solutions away from the event horizon) were obtained in \cite{FKSY}, \cite{FS}.  Recently in \cite{TT} and independently in  \cite{DR2}, \cite{DR3}, \cite{AB} uniform pointwise bounds as well as local energy decay were established for small angular momentum. The aim of this paper is to prove Strichartz estimates under the same assumption of small angular momentum.

  The Kerr geometry in Boyer-Lindquist coordinates is given by
\[ 
ds^2 = g_{tt}dt^2 + g_{t\phi}dtd\phi + g_{rr}dr^2 + g_{\phi\phi}d\phi^2
 + g_{\theta\theta}d\theta^2 
\]
 where $t \in \R$, $r > 0$, $(\phi,\theta)$ are the spherical coordinates 
on $\S^2$ and 
\[
 g_{tt}=-\frac{\Delta-a^2\sin^2\theta}{\rho^2}, \qquad
 g_{t\phi}=-2a\frac{2Mr\sin^2\theta}{\rho^2}, \qquad
 g_{rr}=\frac{\rho^2}{\Delta}
 \]
\[ g_{\phi\phi}=\frac{(r^2+a^2)^2-a^2\Delta
\sin^2\theta}{\rho^2}\sin^2\theta, \qquad g_{\theta\theta}={\rho^2}
\]
with
\[ 
\Delta=r^2-2Mr+a^2, \qquad \rho^2=r^2+a^2\cos^2\theta. 
\]

 Here $M$ represents the mass of the black hole, and $aM$ its angular momentum.
 
 A straightforward computation gives us the inverse of the metric:
\[ g^{tt}=-\frac{(r^2+a^2)^2-a^2\Delta\sin^2\theta}{\rho^2\Delta},
\qquad g^{t\phi}=-a\frac{2Mr}{\rho^2\Delta}, \qquad
g^{rr}=\frac{\Delta}{\rho^2},
\]
\[ g^{\phi\phi}=\frac{\Delta-a^2\sin^2\theta}{\rho^2\Delta\sin^2\theta}
, \qquad g^{\theta\theta}=\frac{1}{\rho^2}.
\]

The case $a = 0$ corresponds to the Schwarzschild space-time.  We shall
subsequently assume that $a$ is small $a \ll M$, so that the Kerr
metric is a small perturbation of the Schwarzschild metric. We let
$\Box_{\mathbf K} $ denote the d'Alembertian associated to the Kerr metric.

In the above coordinates the Kerr metric has singularities at $r = 0$
on the equator $\theta = \pi/2$ and at the roots of $\Delta$, namely
$r_{\pm}=M\pm\sqrt{M^2-a^2}$. The singularity at $r=r_{+}$ is just a coordinate singularity,
and corresponds to the event horizon.  The singularity at $r = r_-$ is
also a coordinate singularity; for a further discussion of its nature,
which is not relevant for our results, we refer the reader to
\cite{Ch,HE}.  To remove the singularities at $r = r_{\pm}$ we
introduce functions $r^*$, $v_{+}$ and $\phi_{+}$ so that (see
\cite{HE})
\[
 dr^*=(r^2+a^2)\Delta^{-1}dr,
\qquad
 dv_{+}=dt+dr^*, 
\qquad 
 d\phi_{+}=d\phi+a\Delta^{-1}dr.
\]

The metric then becomes
\[
\begin{split}
ds^2= &\
-(1-\frac{2Mr}{\rho^2})dv_{+}^2+2drdv_{+}-4a\rho^{-2}Mr\sin^2\theta
dv_{+}d\phi_{+} -2a\sin^2\theta dr d\phi_{+} +\rho^2 d\theta^2 \\
& \ +\rho^{-2}[(r^2+a^2)^2-\Delta a^2\sin^2\theta]\sin^2\theta
d\phi_{+}^2
\end{split}
\]
which is smooth and nondegenerate across the event horizon up to but not including 
$r = 0$. Just like
in \cite{MMTT} and \cite{TT}, we introduce the function
\[
\tv = v_{+} - \mu(r)
\]
where $\mu$ is a smooth function of $r$. In the $(\tv,r,\phi_{+},
\theta)$ coordinates the metric has the form
\[
\begin{split}
ds^2= &\ (1-\frac{2Mr}{\rho^2}) d\tv^2
+2\left(1-(1-\frac{2Mr}{\rho^2})\mu'(r)\right) d\tv dr \\
 &\ -4a\rho^{-2}Mr\sin^2\theta d\tv d\phi_{+} + \Bigl(2 \mu'(r) -
 (1-\frac{2Mr}{\rho^2}) (\mu'(r))^2\Bigr)  dr^2 \\
 &\ -2a\theta (1+2\rho^{-2}Mr\mu' (r))\sin^2dr d\phi_{+} +\rho^2
 d\theta^2 \\
 &\ +\rho^{-2}[(r^2+a^2)^2-\Delta a^2\sin^2\theta]\sin^2\theta
d\phi_{+}^2.
\end{split}
\]

On the function $\mu$ we impose the following two conditions:

(i) $\mu (r) \geq  \rs$ for $r > 2M$, with equality for $r >
{5M}/2$.

(ii)  The surfaces $\tv = const$ are space-like, i.e.
\[
\mu'(r) > 0, \qquad 2 - (1-\frac{2Mr}{\rho^2}) \mu'(r) > 0.
\]
As long as $a$ is small, we can work with  the same 
function $\mu$ as in the case of the Schwarzschild space-time.

 For convenience we also introduce
\[
\tphi = \zeta(r)\phi_{+}+(1-\zeta(r))\phi
\]
where $\zeta$ is a cutoff function supported near the event horizon
and work in the $(\tv,r,\tphi, \theta)$ coordinates which are
identical to $(t,r,\phi,\theta)$ outside of a small neighborhood of
the event horizon.

Given $r_{-} < r_e <2M$ we consider the wave equation 
\begin{equation}
 \Box_{\mathbf K}  u = f 
 \label{boxsinhom}\end{equation}
in the cylindrical region
\begin{equation}
 \M_{R} =  \{ \tv \geq 0, \ r \geq r_e \} 
\label{mr}\end{equation}
with initial data on the space-like surface
\begin{equation}
 \Sigma_R^- =  \M_{R} \cap \{ \tv = 0 \}
\label{mr-}\end{equation}
The lateral boundary of $\M_R$, 
\begin{equation}
 \Sigma_R^+ =   \M_R \cap \{ r = r_e\} 
\label{mr+}\end{equation}
is also space-like, and can be thought of as the exit surface
for all waves which cross the event horizon. 

We define the initial (incoming) energy on $\Sigma_R^-$ as 
\begin{equation}\label{energy1}
 E[u](\Sigma_R^-) =  \int_{\Sigma_R^-} \left( |\partial_r u|^2 +
|\partial_\tv u|^2    + |\ang u|^2 \right) r^2  dr  d\omega  
\end{equation}
the outgoing energy on $\Sigma_R^+$ as 
\begin{equation}\label{energy2}
 E[u](\Sigma_R^+) = \int_{\Sigma_R^+}
 \left(  |\partial_r u|^2 +   |\partial_\tv u|^2    +
|\ang u|^2 \right) r_e^2   d\tv d\omega
\end{equation}
and the energy on an arbitrary $\tv$ slice as
\begin{equation}\label{energy3}
 E[u](\tv_0) = \int_{ \M_R \cap \{\tv = \tv_0\}}
 \left(
|\partial_r u|^2 +   |\partial_\tv u|^2    + |\ang u|^2
\right) r^2  dr  d\omega
\end{equation}

For solutions to the constant coefficient wave equation on $\R\times \R^3$, the
well-known Strichartz estimates state that
\begin{equation}\label{strich}
  \||D_x|^{-\rho_1} \nabla u\|_{L^{p_1}_tL^{q_1}_x} \lesssim
\|\nabla u(0)\|_{L^2}  + \||D_x|^{\rho_2} f\|_{L^{p_2'}_tL^{q_2'}_x}.
\end{equation}
Here the exponents $(\rho_i,p_i,q_i)$ are subject to the
scaling relation
\begin{equation}
\frac{1}p+\frac{3}q = \frac{3}2 -\rho
\label{scalingpq}\end{equation}
and the dispersion relation
\begin{equation}
\frac{1}p + \frac{1}q \leq \frac12, \qquad 2 < p \leq \infty.
\label{dispersionpq}\end{equation}
All pairs $(\rho,p,q)$ satisfying \eqref{scalingpq} and
\eqref{dispersionpq} are called Strichartz pairs. Those for which the
equality holds in \eqref{dispersionpq} are called sharp Strichartz
pairs.  Such estimates first appeared in the seminal works
\cite{Brenner}, \cite{Strichartz1, Strichartz2} and as stated include
contributions from, e.g., \cite{GV}, \cite{P}, \cite{K}, \cite{LiSo},
and \cite{KT}.

If one allows variable coefficients, such estimates are
well-understood locally-in-time.  For smooth coefficients, this was
first shown in \cite{MSS} and later for $C^2$ coefficients in
\cite{Smith} and \cite{T1,T2,T3}.

Globally-in-time, the problem is more delicate.  Even a small, smooth,
compactly supported perturbation of the flat metric may refocus a
group of rays and produce caustics. Thus, constructing a parametrix
for incoming rays proves to be quite difficult. At the same time, one
needs to contend with the possibility of trapped rays at high
frequencies and with eigenfunctions/resonances at low frequencies.

Our main theorem is the following global in time estimate:
\begin{theorem}
\label{Strichartz.theorem}
If $u$ solves $\Box_{\mathbf K}  u = f$ in $\M_R$ then for all nonsharp
Strichartz pairs $(\rho_1,p_1,q_1)$ and $(\rho_2,p_2,q_2)$ we have
\begin{equation}
 E[u](\Sigma_R^+) + \sup_{\tv} E[u](\tv) +  \left\| \nabla u\right\|^2_{L^{p_1}_\tv
\dot H^{-\rho_1,q_1}_x}
 \lesssim E[u](\Sigma_R^-) +
\left\| f \right\|^2_{L^{p_2'}_\tv \dot H^{\rho_2, q_2'}_x }.
\end{equation}
\end{theorem}
Here the Sobolev-type spaces $\dot H^{s,p}$ coincide with the usual
$\dot H^{s,p}$ homogeneous spaces in $\R^3$ expressed in polar
coordinates $(r,\omega)$.

As a corollary of this result one can consider the global solvability
question for the energy critical semilinear wave equation in the
Kerr space. Let 
\[
\Sigma_0 = \{ t = 0 \}
\]
 Note that $\Sigma_0$ is a smooth, spacelike surface.
 
 We now consider the Cauchy problem with initial data on $\Sigma_0$:
\begin{equation}
\left\{ 
\begin{array}{lc} 
\Box_{\mathbf K}  u = \pm u^5  & \text{in } \M 
\cr \cr
u = u_0, \ \tilde T u = u_1 & \text{in } \Sigma_0.
\end{array}
\right.
\label{nonlin}\end{equation}
 Here $\tilde T$ is a vector field that is smooth, everywhere timelike and equals $\partial_\tv$ on
$\Sigma_0$ outside $\M_C$. Observe that we cannot use $\partial_\tv$ on all of 
$\Sigma_0$ since it becomes spacelike inside the ergosphere (i.e. when $g_{tt}>0$).

\begin{theorem}\label{tnlw}
Let $r_e > r_{-}$. Then there exists $\epsilon > 0$ so 
that for each initial data $(u_0,u_1)$ which satisfies
\[
E[u](\Sigma_0) \leq \epsilon
\]
the equation \eqref{nonlin} admits an unique solution
$u$ in the region $\{ r > r_e\}$ which satisfies the bound
\[
E[u](\Sigma_{r_e}) + \| u \|_{\dot H^{s,p}(\{r > r_e\})}  \lesssim
E[u](\Sigma_0)
\]
for all indices $s,p$ satisfying
\[
\frac{4}p = s+\frac12, \qquad 0 \leq s < \frac12.
\]
Furthermore, the solution has a Lipschitz dependence on the initial
data in the above topology. 
\end{theorem}

Theorem \ref{tnlw} follows from Theorem \ref{Strichartz.theorem} exactly like in Section 5 of \cite{MMTT}, hence we will now focus on proving the latter.

{\bf Acknowledgements:} The author would like to thank Daniel Tataru for many useful conversations and suggestions regarding the paper.

\newsection{The Strichartz estimates}

 In this section we will prove Theorem \ref{Strichartz.theorem}. The key to the proof is the approach 
developed in \cite{MT} for small perturbations of the Minkowski metric and \cite{MMTT} for the Schwarzschild metric which allows one to establish global-in-time Strichartz estimates provided that a strong form of the local energy estimates holds. 

 A weaker local energy result has been proved in \cite{TT}. We first improve this to a stronger result that only requires logarithmic losses in the energy norm. We then apply the techniques from \cite{MT} and \cite{MMTT} to obtain Strichartz estimates for all nonsharp exponents.

  Let us first recall the setup and results from \cite{TT}. Let $\tau, \xi, \Phi$ and $\Theta$ be
the Fourier variables corresponding to $t, r, \phi$ and $\theta$, and
\[
p( r, \phi,\tau, \xi, \Phi,\Theta)
=g^{tt}\tau^2+2g^{t\phi}\tau\Phi+g^{\phi\phi}\Phi^2
+g^{rr}\xi^2 +g^{\theta\theta}\Theta^2
\]
 be the principal symbol of $\Box_{\mathbf K} $. On any null geodesic one has
\begin{equation}\label{Ham}
 p(t, r, \phi,\theta,\tau, \xi, \Phi,\Theta)=0.
\end{equation}

 Moreover, all trapped null geodesics in the exterior $r>r_+$ must also satisfy (see \cite{TT} for more details):
 \begin{equation}\label{xi} 
 \xi= 0
 \end{equation}
 \begin{equation}\label{rct} 
  r= \text{constant} \;
 \text{ so that } \; (2r\Delta - (r-M)\rho^2)^2\leq 4a^2 r^2\Delta\sin^2 \theta 
 \end{equation}
 \begin{equation}\label{tauphi} 
  R_a(r,\tau,\Phi)= 0 
 \end{equation}  
  where
  \begin{align*}
  R_a(r,\tau,\Phi)=&(r^2+a^2)(r^3-3Mr^2+a^2r+a^2M)\tau^2 \\
  &{}- 2aM(r^2-a^2)\tau\Phi- a^2(r-M)\Phi^2 
  \end{align*}
 
 By \eqref{Ham} we can bound $\Phi$ in terms of $\tau$,
\begin{equation}
|\Phi| \leq 4 M |\tau|
\end{equation}
For $\Phi$ in this range and small $a$ the polynomial 
$\tau^{-2} R_a(r,\tau,\Phi)$ can be viewed as a small perturbation of 
\[
\tau^{-2} R_0(r,\tau,\Phi) = r^4(r-3M) 
\]
which has a simple root at $r = 3M$. Hence for small $a$ the
polynomial $R_a$ has a simple root close to $3M$, which we denote by
$r_a(\tau,\Phi)$.

   In \cite{TT} the following result is proved:
\begin{theorem}\label{KerrLE}
 Let $u$ solve $\Box_{\mathbf K}  u = f$ in $\M_R$. Then
\begin{equation}
 \|u \|_{LEW_{\mathbf K}}^2 + \sup_{\tilde v} E[u](\tilde v) + E[u](\Sigma_R^+)
 \lesssim E[u](\Sigma_R^-)+ \|f \|_{LEW_{\mathbf K}^{*}}^2.
\end{equation}
\end{theorem}

The exact definition of the norms $LEW_{\mathbf K}$ is not important; what matters is that it is equivalent to $H^1_{t,x}$ on compact sets outside an $O(1)$ neighborhood of $r=3M$ and is degenerate on the trapped set (described above). Near infinity (say for $r>4M$) the norm $LEW_{\mathbf K}$ is defined on dyadic regions as follows:
\begin{equation}\label{LEinfty}
 \|u\|_{LEW_{\mathbf K}} = \sup_{j \in \Z}
2^{-\frac{j}2}\|\nabla_{t, x} u\|_{L^2 (\R \times \{|x|\in [2^{j-1}, 2^j]\})} + 2^{-\frac{3j}2}\|u\|_{L^2 (\R \times \{|x|\in [2^{j-1}, 2^j]\})}
\end{equation}

 The norm $LEW_{\mathbf K}^{*}$ is the dual of $LEW_{\mathbf K}$; in particular, near infinity it will be defined as  
 \begin{equation}
 \|f\|_{LEW_{\mathbf K}^*}= \sum_{j \in \Z}
2^{\frac{j}2}\|f\|_{L^2 (\R \times \{|x|\in [2^{j-1}, 2^j]\})}
\label{le*m}\end{equation}

 We first improve Theorem \ref{KerrLE} around $r=3M$. We will work with the usual Boyer-Lindquist coordinates $(t, r, \phi, \theta)$ since we are only interested in improving the estimate around $r=3M$. In these coordinates we can write
\begin{equation}
\rho^2 \sqrt{\Delta} \Box_{\mathbf K}  \frac{1}{\sqrt{\Delta}} = L_{K} =
 \Delta \partial_r^2 + (-\frac{(r^2+a^2)^2}{\Delta}\partial_t^2 -a\frac{2Mr}{\Delta}\partial_t
 \partial_{\phi} -\frac{a^2}{\Delta}\partial_{\phi}^2 + L_a)  + V(r), 
\label{KerrRW}\end{equation}  
where 
\[
 L_a = a^2 \sin^2 \theta \partial_t^2 + \frac{1}{\sin^2 \theta}\partial_{\phi}^2 + \frac{1}{\sin \theta}\partial_{\theta} (\sin \theta \partial_{\theta})
\]
and
\[
 V(r) =  \sqrt{\Delta}\partial_{r}(\Delta\partial_{r}\frac{1}{\sqrt{\Delta}})
\]

 We can now take the Fourier transform in $t$ (this is allowed since by Theorem \ref{KerrLE} since $u$ is a tempered distribution in $t$) and $\phi$, and expand $\mathcal{F}_t L_a$ (which is an elliptic operator on $\S^2$) as the countable sum of its eigenvectors and corresponding eigenvalues $\lambda_a^2$ (note that $\lambda_0$ correspond to the usual spherical harmonics). This is possible since the operators $\partial_\phi$ and $L_a$ commute. We are left with the ordinary differential equation
 \begin{equation}\label{psKeq}
  (\Delta \partial_r^2 + V_{\lambda_a, \tau, \Phi}(r))w(r) = g
 \end{equation}
where
\begin{equation}\label{Vdiff}
V_{\lambda_a,\tau, \Phi}(r) = \tau^2 W(r - r_a (\tau, \Phi)) - \lambda_a^2 + V.
\end{equation}
and $W(0)=W'(0)=0$, $W''(0)>0$.

 Let $\gamma_0: \R \to \R^+$ be a smooth increasing function so that
\[
\gamma_0(y) = \left\{ \begin{array}{cc} 1& y < 1 \cr y & y \geq 2.
  \end{array}
\right.
\]

Let $\gamma_1: \R^+ \to \R^+$ be a smooth increasing function so that
\[
\gamma_1(y) = \left\{ \begin{array}{cc} y^\frac12 & y < 1/2 \cr 1 & y \geq 1.
  \end{array}
\right.
\]

Let $\gamma: \R^2 \to \R^+$ be a smooth function with the following
properties:
\begin{equation}\label{fnpdo}
\gamma(y,z) = \left\{ \begin{array}{cc} 1 & z < C \cr 
\gamma_0(y)  &  y< \sqrt{z/2},\ z \geq C \cr
z^\frac12 \gamma_1(y^2/z) & y \geq   \sqrt{z/2}, \   z   \geq C
  \end{array}
\right.
\end{equation}
where $C$ is a large constant. In the sequel $z$ is either a discrete
parameter or very large (see Remark \ref{roleoflambda}), so the lack of smoothness at $z=C$ is of no consequence.
 
 We define the symbol
\[
b_{ps}(r, \tau, \xi, \Phi, \lambda_a) = \gamma(-\psi(\frac{\lambda_a}{\tau})\ln((r-r_a(\tau, \Phi))^2 +\lambda_a^{-2} \xi^2), \ln \lambda_a),
\]
and its inverse 
\[
b_{ps}^{-1}(r, \tau, \xi, \Phi, \lambda_a) = 
\frac{1}{\gamma(-\psi(\frac{\lambda_a}{\tau})\ln((r-r_a(\tau, \Phi))^2 +\lambda_a^{-2} \xi^2), \ln \lambda_a)}.
\]
 Here $\psi:\R\to\R$ is a smooth cutoff such that 
\[
 \psi(y) = \left\{ \begin{array}{cc} 1 & y \in [-4, -2]\cup [2, 4] \cr 0 & y\in (-\infty, 8]\cup [-1, 1] \cup [8, \infty)
 \end{array}
\right. 
\] 
 The role of $\psi$ is to make sure that $b=1$ when $\tau \ll \lambda_a$ and $\tau \gg \lambda_a$.
 
 Observe that, as opposed to their Schwarzschild counterparts $a_{ps}$ and  $a_{ps}^{-1}$ defined in \cite{MMTT}, the symbols $b_{ps}$ and $b_{ps}^{-1}$ will depend on $\tau$ and $\Phi$.

 We note that if $\lambda_a$ is small then $b_{ps}$ and $b_{ps}^{-1}$ both equal $1$, while if
$\lambda_a$ is large then they satisfy the bounds
\begin{equation}
\begin{split}
  1  \leq b_{ps}(r, \tau, \xi, \Phi, \lambda_a) \leq b_{ps}(r, \tau, 0, \Phi, \lambda_a)\leq (\ln
  \lambda_a)^\frac12,
  \\
  (\ln
  \lambda_a)^{-\frac12}\leq  b_{ps}^{-1}(r, \tau, 0, \Phi, \lambda_a) \leq b_{ps}^{-1}(r, \tau, \xi, \Phi, \lambda_a) 
  \leq 1.
\end{split}
\label{bpsbds}\end{equation}

 We also observe that the region where $y^2 > z$ corresponds to $(r-r_a(\tau, \Phi))^2
+\lambda_a^{-2} \xi^2 < e^{-\sqrt{ \ln \lambda_a}}$. Thus differentiating the two symbols we obtain the following bounds
\begin{align}
|\partial_{r}^\alpha \partial_\xi^\beta \partial_{\lambda_a}^\nu \partial_\tau^\eta
b_{ps}(r, \tau, \xi, \Phi, \lambda_a)| \leq c_{\alpha,\beta,\nu} (1+|\ln((r-r_a(\tau, \Phi))^2 +\lambda_a^{-2} \xi^2)|)\lambda_a^{-\beta-\nu-\eta} \\
((r-r_a(\tau, \Phi))^2 +\lambda_a^{-2} \xi^2 + e^{-\sqrt{ \ln \lambda_a}} )^{-\frac{\alpha+\beta+\eta}2},
\label{bpsbd}\end{align}
respectively 
\begin{align}
|\partial_{r}^\alpha \partial_\xi^\beta \partial_{\lambda_a}^\nu \partial_\tau^\eta
b_{ps}^{-1}(r, \tau, \xi, \Phi, \lambda_a)| \leq c_{\alpha,\beta,\nu}
b_{ps}^{-2}(r, \tau, \xi, \Phi, \lambda_a) (1+|\ln((r-r_a(\tau, \Phi))^2 +\lambda_a^{-2} \xi^2)|) \\ \lambda_a^{-\beta-\nu-\eta}((r-r_a(\tau, \Phi))^2 +\lambda_a^{-2}
\xi^2 + e^{-\sqrt{ \ln \lambda_a}} )^{-\frac{\alpha+\beta+\eta}2},
\label{bps1bd}\end{align}
when $\alpha + \beta + \nu + \tau> 0$ and $\ln \lambda \geq C$.
These show that we have a good operator calculus for the corresponding pseudodifferential operators.
In particular in terms of the classical symbol classes we have
\begin{equation}\label{apssymb}
b_{ps}, b_{ps}^{-1} \in S^{\delta}_{1- \delta, 0}, \qquad \delta > 0.
\end{equation}

By \eqref{bpsbd} and \eqref{bps1bd} one easily sees that these operators are approximate inverses. More precisely for small $\lambda_a$, $\ln \lambda_a < C$, they are both the identity, while
for large $\lambda_a$ 
\begin{equation}\label{almostid}
\| b_{ps}^w(\lambda_a) (b_{ps}^{-1})^w(\lambda_a)-I\|_{L^2 \to L^2} \lesssim
\lambda_a^{-1} e^{\sqrt{\ln \lambda_a}}, \qquad \ln \lambda_a \geq C.
\end{equation}
Choosing $C$ large enough we insure that the bound above
is always much smaller than $1$.

\begin{remark}\label{roleoflambda} 
 The role played by $\lambda_a$ changes from the proof of Theorems \ref{theorem.3} to the proof of the Strichartz estimates. Since all of our $L^2$ estimates admit orthogonal decompositions with respect to the eigenfunctions associated to $\lambda_a$, it suffices for Theorems \ref{theorem.3} to fix $\lambda_a$ and work with the operators $b_{ps}^w(\lambda_a)$. However, in the proof of the Strichartz estimates we need kernel bounds for operators of the form $B_{ps}$, which is why we think of $\lambda_a$ as a fourth Fourier variable (besides $\tau$, $\Phi$ and $\xi$) and track the symbol regularity with respect to $\lambda_a$ as well. Of course, this is meaningless for $\lambda_a$ in a compact set; only 
the asymptotic behavior as $\lambda_a \to \infty$ is relevant.
\end{remark} 

We can now introduce the Weyl operators
 \[
B_{ps} = \sum_{\lambda_a} b_{ps}^w(\lambda_a) \Pi_{\lambda_a},
\]
respectively
\[
B_{ps}^{-1} = \sum_{\lambda_a} (b_{ps}^{-1})^w(\lambda_a) \Pi_{\lambda_a}.
\]
 where $\Pi_{\lambda_a}$ are the spectral projectors on the eigenspaces of $L_a$ determined by $\lambda_a$. 
 
We use these two operators in order to define the improved local
energy norms around $r=3M$. Let $I $ be a small neighborhood of $3M$ 
(which does not depend on $a$). We say that $\tilde u \in LEK_{ps}$ if $\tilde u : \R \times I \times \S^2 \to \R$ and 
\begin{equation}
  \|\tilde u\|_{LEK_{ps}} =  \|B^{-1}_{ps}\tilde u\|_{H^1} \approx 
\| A_{ps}^{-1} \nabla\tilde u\|_{L^2} < \infty,
\end{equation}

Similarly $\tilde f \in LEK_{ps}^*$ if $\tilde f :  \R \times I \times \S^2 \to \R$ and 
\begin{equation}
\|\tilde f \|_{LEK_{ps}^*} = \| B_{ps}\tilde f\|_{L^2} < \infty.
\end{equation} 

 Let $\chi(r)$ be a cutoff function supported on $I$ which equals $1$ 
on a smaller neighborhood of size $O(1)$ near $3M$, and $u: \M_R \to \R$. We say that $u \in LE_{\mathbf K}$ if $\chi u$ is the restriction of some $\tilde u \in LEK_{ps}$ on $\M_R$, and
\begin{equation}
  \| u\|_{LE_{\mathbf K}} = \inf_{\tilde u \mid_{\M_R} = \chi u}\|\tilde u\|_{LEK_{ps}} + \|(1-\chi) u\|_{LEW_{\mathbf K}} 
\end{equation}  

 Similarly if $f: \M_R \to \R$, then $f \in LE_{\mathbf K}^*$ if $\chi f$ is the restriction of some $\tilde f \in LEK_{ps}^*$ on $\M_R$, and
\begin{equation}
\| f\|_{LE_{\mathbf K}^*} = \inf_{\tilde f \mid_{\M_R} = \chi u} \|\chi \tilde f\|_{LEK_{ps}^*} + \|(1-\chi) f\|_{LEW_{\mathbf K}^{*}} 
\end{equation}

Our improved local energy estimate reads:
\begin{theorem} \label{theorem.3} For all functions $u$ which solve
  $\Box_{\mathbf K} u = f$ in $\M_R$ we have
\begin{equation*}
\sup_{\tv > 0}  E[u](\tv) +  E[u](\Sigma_R^+) +
\| u \|_{LE_{\mathbf K}}^2  \lesssim E[u](\Sigma_R^-) + \| f \|_{LE_{\mathbf K}^*}^2.
\end{equation*}
\end{theorem}

\begin{proof}
 We start with the following estimate near $r=3M$, which is the equivalent of Proposition 3.3 from \cite{MMTT}:
 
 \begin{proposition}
a) Let $u$ be a function supported in $\{ 5M/2 < r < 5M \}$
which solves $\Box_{\mathbf K} u = f$. Then
\begin{equation}\label{Kerrpsb}
 \| u\|_{LEK_{ps}}^2\lesssim \| f\|^2_{LEK_{ps}^*}.
\end{equation}

b) Let $f \in {LEK_{ps}}^*$ be supported in $\{ 11M/4 < r < 4M \}$. Then there
is a function $u$ supported in $\{ 5M/2 < r < 5M \}$ so that
\begin{equation}
\sup_t E[u] + \| u\|_{LEK_{ps}}^2 
+ \| \Box_{\mathbf K} u -f\|_{LE_{\mathbf K}^{1*}}^2
\lesssim \| f\|^2_{ LEK_{ps}^*}.
\label{Kerrppsb} \end{equation}
\label{Kerrppsbprop}\end{proposition}
\begin{proof} 
 
 We start with part (a). By Plancherel's formula and the fact that 
 \begin{equation}\label{Phi<lambda_a}
 \Phi < \lambda_a
 \end{equation}
 \eqref{Kerrpsb} will follow if we can prove that
\begin{equation}
  \| \partial_r w(r) \|_{L^2} + (|\tau|+|\lambda_a|)\|(b_{ps}^{-1})^w(\tau, \Phi, \lambda_a) w\|_{L^2}  \lesssim \| b_{ps}^w(\tau, \Phi, \lambda_a) g\|_{L^2}.
\label{Kerrpsbhard}\end{equation} 
for any fixed $\tau$, $\Phi$ and $\lambda_a$ (here $b_{ps}^w(\tau, \Phi, \lambda_a)$ is the one dimensional pseudodifferential operator obtained from $b_{ps}$ by fixing $\tau$, $\Phi$ and $\lambda_a$.)

  Depending on the relative sizes of $\lambda_a$, $\tau$ and $\Phi$ and taking into account \eqref{Phi<lambda_a} we consider several cases. In the easier cases it suffices to replace the
bound \eqref{Kerrpsbhard} with a simpler bound
\begin{equation}
  \| \partial_{r} w \|_{L^2} + (|\tau|+|\lambda_a|)\|w\|_{L^2} \lesssim \|g\|_{L^2}.
\label{Kerrpsbeasy}\end{equation}

{\bf Case 1:} $\lambda_a,\tau, \Phi \lesssim 1$.
Then $V_{\lambda_a,\tau, \Phi}(r) \approx 1$. We solve \eqref{psKeq} as a Cauchy problem with data on one side; namely, if
\[
 E[w](r) = \Delta(\partial_ r w)^2 + w^2
\] 
then an easy computation shows that

\[
 \frac{d}{dr}E[w](r) \lesssim E[w](r) + g^2
\]
 By Gronwall's inequality and the fact that $E[w](r) = 0$ away from the photonsphere 
  we obtain the pointwise bound
\[
|w|+|\partial_{r} w| \lesssim \|g\|_{L^2}
\]
which easily implies \eqref{Kerrpsbeasy}.

{\bf Case 2:} $\Phi, \lambda_a \ll \tau$.  Then $V_{\lambda_a,\tau, \Phi}(r) \approx
\tau^2$, $\partial_r V_{\lambda_a,\tau, \Phi}(r) \approx \tau^2$ for $r$ in a compact set; therefore \eqref{psKeq} is hyperbolic in nature.  Hence we can solve \eqref{psKeq} as a Cauchy
problem with data on one side; namely, if
\[
 E[w](r) = \Delta (\partial_r w)^2 + V_{\lambda_a,\tau,\Phi}(r)w^2
\] 
then an easy computation shows that

\[
 \frac{d}{dr}E[w](r) \lesssim E[w](r) + g^2
\]
 By Gronwall's inequality and the fact that $E[w](r) = 0$ away from the photonsphere 
 we obtain the pointwise bound

\[
\tau |w|+|\partial_r w| \lesssim \|g\|_{L^2}
\]
which implies \eqref{Kerrpsbeasy}.

{\bf Case 3:} $\lambda_a \gg \tau$. Then $V_{\lambda_a,\tau, \Phi}(r) \approx
-\lambda_a^2$ for $r$ in a compact set; therefore \eqref{psKeq} is
elliptic.  Then we solve \eqref{psKeq} as an elliptic problem with
Dirichlet boundary conditions on a compact interval and obtain
\[
\lambda_a^{\frac32}\|w\|_{L^2}+\lambda_a^{\frac12}\|\partial_r w\|_{L^2} \lesssim \|g\|_{L^2}
\]
which again gives \eqref{Kerrpsbeasy}.

{\bf Case 4:} $\lambda_a \approx \tau \gg 1$.  In this case
\eqref{Kerrpsbeasy} no longer holds. However, we can now use Proposition 3.4 from \cite{MMTT} which states
\begin{proposition}\label{ODE}
Let $W$ be a smooth function satisfying  $W(0)=W'(0)=0$, $W''(0)>0$,
and $|\epsilon| \lesssim 1$. Let $w$ be a solution of the ordinary differential equation
\[
(\partial_\rs^2 + \lambda^2(W(\rs) +\epsilon)) w(\rs) = g
\]
supported near $\rs = 0$. Then we have
\begin{equation}
\label{psahard}
\| \partial_{\rs} u \|_{L^2} + \lambda \| (a_{ps}^{-1})^w(\lambda)  u\|_{L^2}
\lesssim \| a_{ps}^w(\lambda)  g\|_{L^2} ,
\end{equation}
 Here the symbols $a_{ps}(\lambda)$ and $(a_{ps}^{-1})(\lambda)$ are defined as follows
\[
a_{ps}(\lambda)(\rs,\xi) = \gamma(-\ln(\rs^2 +\lambda^{-2} \xi^2), \ln \lambda),
\]
\[
a_{ps}^{-1}(\lambda)(\rs,\xi) = 
\frac{1}{\gamma(-\ln(\rs^2 +\lambda^{-2} \xi^2), \ln \lambda)}.
\] 
\end{proposition} 

 We can now apply Proposition \ref{ODE} with $\rs = r-r_a(\tau,\Phi)$, $W$ as in \eqref{Vdiff} and with $b_{ps}^w(\tau, \Phi, \lambda_a)$ and $(b_{ps}^{-1})^w(\tau, \Phi, \lambda_a)$ replacing $a_{ps}^w(\lambda)$ and $(a_{ps}^{-1})^w(\lambda)$ respectively. Note that there the extra $\Delta$ coefficient in front of $\partial_r^2$ in \eqref{psKeq} plays no role, since $\Delta \approx \partial_r \Delta \approx 1$ near $r=3M$.

 Part (b) follows now from part (a) exactly like in Proposition 3.3 from \cite{MMTT} with $\lambda_a$ replacing the spherical harmonics $\lambda$ and $\rs = r-r_a(\tau,\Phi)$.
\end{proof}

\end{proof}

We now repeat the arguments in Section 4 of \cite{MMTT} to turn the improved local energy estimate, Proposition \ref{Kerrppsbprop} into Strichartz estimates, Theorem \ref{Strichartz.theorem}. The only (minor)
 differences appear in proving Case II of Proposition 4.10, which we will settle below. We thus need to prove
 \begin{proposition}\label{Kerrps}
 For $u$ supported in $\{ 5M/2 < r < 5M\}$ we have
\[
\|\nabla u\|_{L^p_\tv H^{-\rho,q}}^2 \lesssim
E[u](0)+\|u\|_{LEK_{ps}}^2 + \|\Box_K u\|_{LEK^*_{ps}}^2.
\]
\end{proposition}
\begin{proof}
Clearly the operator 
$\Box_K$ can be replaced by $L_{K}$. The potential $V$ can
be neglected due to the straightforward bound
\[
\|V u\|_{LE^*_{ps}} \lesssim \|u\|_{LE_{ps}}.
\]
Indeed, for $u_{\lambda_a}$ the eigenfunction corresponding to the eigenvalue $\lambda_a^2$ we have
\[
\|V u_{\lambda_a}\|_{LEK^*_{ps}} \lesssim |\ln (2+\lambda_a)|^\frac12 \|u_{\lambda_a}\|_{L^2}
\lesssim \lambda_a |\ln (2+\lambda_a)|^{-\frac12} \|u_{\lambda_a}\|_{L^2} \lesssim
\|u_{\lambda_a}\|_{LE_{ps}}
\]

 We introduce the auxiliary function
\[
\psi = B_{ps}^{-1} u.
\]
By the definition of the $LEK_{ps}$ norm we have
\begin{equation}
\| \psi \|_{H^1} \lesssim \|u\|_{LEK_{ps}}.
\label{h1Kpsi}\end{equation}
We also claim that 
\begin{equation}
\|L_K \psi\|_{L^2}  \lesssim \| u\|_{LEK_{ps}} + \| L_K u\|_{L^2}.
\label{Kppsi}\end{equation}
Since $B_{ps}^{-1}$ is $L^2$ bounded, this is a consequence of the 
commutator bound
\[
[  B_{ps}^{-1}, L_K ] : LEK_{ps} \to L^2,
\]
or equivalently
\begin{equation}
[  B_{ps}^{-1}, L_K] B_{ps}: H^1 \to L^2.
\label{Kcmt}\end{equation}
It suffices to consider the first term in the symbol calculus, as the
remainder belongs to $OPS^\delta_{1-\delta,\delta}$, mapping $H^\delta$ to
$L^2$ for all $\delta > 0$. The symbol of the first term is 
\[
q(\tau, \Phi, \xi,r,\lambda_a) = \{ b_{ps}^{-1}(\lambda_a), \Delta \xi^2 -\frac{(r^2+a^2)^2}{\Delta}\tau^2 -a\frac{2Mr}{\Delta}\tau\Phi -\frac{a^2}{\Delta}\Phi^2 + \lambda_a^2 \}
b_{ps}(\lambda_a)
\]
and a-priori we have $q \in S^{1+\delta}_{1-\delta,\delta}$.
For a better estimate we compute the Poisson bracket
\[
\begin{split}
q(\tau, \Phi, \xi,r,\lambda_a)  = &\ b_{ps}^{-1}(\lambda_a) \gamma_y(y,\ln \lambda_a) 
\psi(\frac{\Phi}{\tau}) \\ &\ \frac{ 4 \xi (r-r_a(\tau, \Phi)) - 2 \xi \lambda_a^{-2}(\tau^2 W(r-r_a(\tau, \Phi)) + 2(r-M)\xi^2}{(r-r_a(\tau, \Phi))^2 + \lambda_a^{-2}\xi^2}
\end{split}
\]
where $y = (r-r_a(\tau, \Phi))^2 + \lambda_a^{-2}\xi^2$ and $W\approx r^2$. The first three factors on the
right are bounded. The fourth is bounded by $\max \{\lambda_a, \tau \}$ by Cauchy Schwarz and the fact that $q$ is supported in $|\xi| \lesssim \lambda_a$.  Hence we obtain $q \in
(\tau + \lambda_a) S^{0}_{1-\delta,\delta}$, and the commutator bound \eqref{Kcmt} follows.

Given \eqref{h1Kpsi} and \eqref{Kppsi}, we localize as before $\psi$ to time intervals of unit length and then apply the local Strichartz estimates. By summing over these strips we
obtain
\[
\| \nabla \psi \|_{L^p H^{-\rho,q}} \lesssim \| u\|_{LEK_{ps}} + \|
L_K u\|_{L^2}
\]
for all sharp Strichartz pairs $(\rho,p,q)$.

To  return to $u$ we invert $B_{ps}^{-1}$,
\[
u = B_{ps} \psi + (I - B_{ps} B_{ps}^{-1}) u.
\]
The second term is much more regular, since by \eqref{almostid} we have
\[
 I - B_{ps} B_{ps}^{-1} \in S^{-1 + \delta}_{1- \delta, 0}, \qquad \delta > 0.
\]
This leads to 
\[
\| \nabla (I - B_{ps} B_{ps}^{-1}) u\|_{L^2 H^{1-\delta}} \lesssim
\| u \|_{LEK_{ps}}, \qquad \delta > 0;
\]
therefore all the Strichartz estimates are satisfied simply by Sobolev 
 embeddings. 

For the main term $B_{ps}\psi$ we take advantage of the fact that
we only seek to prove the nonsharp Strichartz estimates for $u$.
The  nonsharp Strichartz estimates for $\psi$ are obtained from the
sharp ones via Sobolev embeddings,
\[
\|\nabla \psi\|_{H^{-\rho_2,q_2}} \lesssim \|\nabla
\psi\|_{H^{-\rho_1,q_1}}, \qquad \frac{3}q_2 + \rho_2 = \frac{3}q_1 +
\rho_1, \quad \rho_1 < \rho_2.
\]
To obtain the nonsharp estimates for $u$ instead, we need a
slightly stronger form of the above bound, namely

\begin{lemma}
Assume that $(\rho_1, p_1, q_1)$, $(\rho_2, p_2, q_2)$ are Strichartz pairs with $p_1 < p_2$, $q_1 < q_2 < \infty$. Then
\begin{equation}\label{Kps.Sobolev}
\|B_{ps} w \|_{L^{p_2}H^{-\rho_2,q_2}} \lesssim \| w\|_{L^{p_1}H^{-\rho_1,q_1}}, \qquad 
\frac{3}q_2 + \rho_2 = \frac{3}q_1 + \rho_1.
\end{equation}
\end{lemma}

\begin{proof}
We need to prove that the operator
\[
\tilde B = Op^w(\xi^2+\lambda_0^2+1)^{-\frac{\rho_2}2} B_{ps} 
Op^w(\xi^2+\lambda_0^2+1)^{\frac{\rho_1}2}
\] 
maps $L^{q_1}$ into $L^{q_2}$. The principal symbol of 
$\tilde B$ is 
\[
\tilde b(r,\tau, \xi,\Phi, \lambda_a) = (\xi^2+\lambda_0^2+1)^{\frac{\rho_1-\rho_2}2}
b_{ps}(\rs,\xi,\lambda_a),
\]
and by the pdo calculus the remainder is easy to estimate,
\[
\tilde B - \tilde b^w \in OPS^{ \rho_1-\rho_2 -1 +\delta}_{1-\delta, 0}, \qquad \delta > 0.
\]
The conclusion of the lemma will follow from the
Hardy-Littlewood-Sobolev inequality (applied separately for $t$, $r$ and $\omega$)
if we prove a suitable pointwise bound on the kernel $K$ of $b_0^w$, namely
\begin{equation}
|K(t_1, r_1,\omega_1, t_2, r_2,\omega_2)| \lesssim |t_1-t_2|^{-1 + \frac{1}{p_1} -\frac{1}{p_2}}(|r_1-r_2| 
|\omega_1-\omega_2|^2)^{-1 + \frac{1}{q_1} -\frac{1}{q_2}}.
\label{KerrKbd}\end{equation}
For fixed $r$ we consider a smooth dyadic partition of unity 
in frequency as follows:
\[
1  = (1-\chi_{\{\tau \approx \lambda_a \}})\chi_{\{|\xi| > \lambda_a\}}\chi_{\{\mu < 1\}}  
+ \chi_{\{\tau \approx \lambda_a \}}\sum_{\mu\geq 1 \, \text{dyadic}} \chi_{\{\lambda_a \approx \mu\}}
 \sum_{\nu = 0}^\mu \chi_{\{|\xi| \approx \nu\}}
\]

This leads to a similar decomposition for $\tilde b$, namely
\[
\tilde b = {\tilde b}_{00} + \sum_{\mu\geq 1} \sum_{\nu = 0}^\mu
\tilde b_{\mu \nu}.
\]
In the region $\{|\xi| \gtrsim \lambda\}\cup \{\mu < 1\}\cup \{\tau \gg \lambda \}\cup \{\tau \ll \lambda \}$ the symbol $\tilde b$ is of class $S^{\rho_1-\rho_2}$, which yields a kernel bound  for $\tilde b_{00}$
of the form 
\[
|K_{00}(t_1, r_1,\omega_1,t_2, r_2,\omega_2)| \lesssim (|t_1-t_2| + |r_1-r_2| +
|\omega_1-\omega_2|)^{-3 + \rho_2 - \rho_1}.
\]
The symbols of $\tilde b_{\mu \nu}$ are supported in $\{ |\xi| \approx \nu, \
\lambda \approx \tau \approx \mu \}$, are smooth on the same scale and have
size at most $\ln (\nu^{-1} \mu) \mu^{\rho_1-\rho_2}$. Clearly on their support we also have $\lambda_0 \approx \lambda_a \approx \mu$.
Hence after integration by parts we get that their kernels satisfy bounds 
of the form
\begin{align*}
& |K_{\mu,\nu} (t_1, r_1,\omega_1, t_2, r_2,\omega_2)|\\
\lesssim &\ln (\nu^{-1}\mu)  \mu^{\rho_1-\rho_2} \nu (|r_1-r_2|\nu + 1)^{-N} \mu^2 (|\omega_1-\omega_2|\mu + 1)^{-N}(|t_1-t_2|\mu + 1)^{-N} \\
\lesssim & |t_1-t_2|^{-1 + \frac{1}{p_1} -\frac{1}{p_2}} \ln (\nu^{-1}
\mu)  \mu^{\frac{3}{q_1}-\frac{3}{q_2}} \nu (|r_1-r_2|\nu + 1)^{-N} \mu^2
(|\omega_1-\omega_2|\mu + 1)^{-N}
\end{align*}
for all nonnegative integers $N$. Then \eqref{KerrKbd} follows after
summation.
\end{proof}

 \end{proof}

\end{document}